\documentclass[onecolumn,journal,10pt]{IEEEtran}
\usepackage{amsfonts,amsmath,amssymb,amsthm,caption,algorithmic}
\usepackage{amsbsy}
\usepackage{tikz}
\usepackage{graphicx,multirow,bm}
\usepackage{color}
\usepackage[normal]{threeparttable}

\makeatletter
\newtheoremstyle{mythm}{3pt}{3pt}{}{16pt}{\bfseries}{:}{.5em}{}
\theoremstyle{mythm}
\newtheorem{theorem}{Theorem}
\newtheorem{example}{Example}
\newtheorem{definition}{Definition}

\newtheorem{corollary}{Corollary}
\newtheorem{lemma}{Lemma}
\newtheorem{construction}{Construction}

\DeclareMathOperator{\lcm}{lcm}

\newcommand{\cA}{\mathcal{A}}
\newcommand{\cB}{\mathcal{B}}

\newcommand{\cR}{\mathcal{R}}
\newcommand{\cS}{\mathcal{S}}

\newcommand{\cU}{\mathcal{U}}
\newcommand{\cV}{\mathcal{V}}
\newcommand{\cW}{\mathcal{W}}

\newcommand{\Z}{\mathbb{Z}}
\newcommand{\F}{\mathbb{F}}

\usepackage{pgf}
\usepackage{tikz}
\usetikzlibrary{matrix,chains,positioning,arrows,calc,decorations,plotmarks,patterns, fit,backgrounds}
\usepackage{pgfplots}
\usetikzlibrary{external}                       
\usepackage{tkz-graph}
\pgfplotsset{compat=1.3}
\tikzstyle{help lines}=[black!20,dashed]

\pgfplotscreateplotcyclelist{mylist}{red,blue,black,yellow,brown}

\pgfdeclarepatternformonly{my north east lines}{\pgfqpoint{-1pt}{-1pt}}{\pgfqpoint{8pt}{8pt}}{\pgfqpoint{6pt}{6pt}}%
{
	\pgfsetlinewidth{0.4pt}
	\pgfpathmoveto{\pgfqpoint{0pt}{0pt}}
	\pgfpathlineto{\pgfqpoint{6.1pt}{6.1pt}}
	\pgfusepath{stroke}
}
\definecolor{light_gray}{rgb}{0.6,0.6,0.6}
\definecolor{awgray}{rgb}{0.7,0.7,0.7}
\definecolor{awgray_dark}{rgb} {0.4,0.4,0.4}

\tikzset{
	>=stealth',
	mycircle/.style={circle, draw=gray, very thick, text width=.1em, minimum height=1.5em, text centered},
	mycircle_small/.style={circle,draw=awgray_dark,fill = awgray_dark, inner sep=0,minimum size=.6em},
	mycircle_small_black/.style={circle,draw=black,fill = black, inner sep=0,minimum size=.6em},
	mybox/.style={rectangle,rounded corners,draw=black, thick,text width=1em,minimum height=4em,minimum width=4em,text centered},
	mybox_small/.style={rectangle,rounded corners,draw=black, thick,text width=1em,minimum height=2em,minimum width=2em,text centered},
	mybox_vec/.style={rectangle,rounded corners,draw=black, thick,text width=1em,minimum height=0.7em, minimum width=4em,text centered},
	mybox_vec_short/.style={rectangle,rounded corners,draw=black, thick,text width=1em,minimum height=0.7em, minimum width=2em,text centered},
	pil/.style={->, thick, shorten <=2pt, shorten >=2pt,},
}

\begin{document}
\title{On optimal weak algebraic manipulation detection codes
and weighted external difference families
\author{Minfeng Shao and Ying Miao}
\thanks{M. Shao and Y. Miao are with the Graduate School of Systems and Information Engineering,
University of Tsukuba, Tennodai 1-1-1, Tsukuba 305-8573, Japan
(e-mail: minfengshao@gmail.com, miao@sk.tsukuba.ac.jp).
}
 }
\date{}
\maketitle

\vspace{0.1in}
\begin{abstract}
This paper provides a combinatorial characterization of weak algebraic manipulation detection (AMD)
codes via a kind of generalized external difference families called bounded
standard weighted external difference families (BSWEDFs).
By means of this characterization, we improve a known lower bound on the maximum
probability of successful tampering for the adversary's all possible strategies in weak AMD codes.
We clarify the relationship between weak AMD codes and BSWEDFs with various properties.
We also propose several explicit constructions for BSWEDFs, some of which can generate new optimal
weak AMD codes.
\end{abstract}

\section{Introduction}

Algebraic manipulation detection (AMD) codes were first introduced by
Cramer \textit{et al.} \cite{CDFPW} to convert linear secret sharing schemes into
robust secret sharing schemes and build nearly optimal robust fuzzy extractors.
For those cryptographic applications, AMD codes received much attention and
were further studied in \cite{AS,CFP,CPX}. Generally speaking, for AMD codes,
we consider two different settings: the
adversary has full knowledge of the source (the \textit{strong model}) and the adversary
has no knowledge about the source (the \textit{weak model}). In the viewpoint of combinatorics,
AMD codes were proved to be closely related with various kinds of external difference
families for both strong and weak models by Paterson and Stinson \cite{PS}. In the literature,
optimal AMD codes in the strong model and their corresponding generalized external
difference families received the most attention (see \cite{BJWZ,HP2018,JL,LNC,PS,MS,WYF,WYFF},
and the references therein), while relatively little was known about AMD codes
under the weak model.

In this paper, we focus on weak AMD codes. In \cite{PS},
Paterson and Stinson first derived a theoretic bound on the maximum probability
of successful tampering for weak AMD codes.
Very recently, Huczynska and Paterson \cite{HP} characterized
the optimal weak AMD codes with respect to the Paterson-Stinson bound
by weighted external difference families. Natural questions arise from
the Paterson-Stinson bound and the corresponding characterization are:
(i) Whether the Paterson-Stinson bound is always
tight; (ii) If not, what are the equivalent combinatorial structures for
those optimal weak AMD codes not having been characterized by the
characterization in \cite{HP}.

To answer these questions, in this paper, we further study
the relationship between weak AMD codes and weighted external
difference families. Firstly, we define a new type of weighted external
difference families which are proved equivalent with weak AMD codes.
By means of this combinatorial characterization of weak AMD codes:
(1) We improve the known lower bound on the maximum probability
of successful tampering for the adversary's all possible strategies; (2) We
derive a necessary condition for the Paterson-Stinson bound to
be achieved;
(3) We determine the exact combinatorial structure for a weak AMD code to be
optimal, when the Paterson-Stinson bound is not achievable.
In this way, some weak AMD codes which have not been identified to be
$R$-optimal previously now can be identified to be in fact $R$-optimal.
Secondly, we show the relationships between this new type of weighted external
difference families and other types of external difference families. Finally,
we exhibit several explicit constructions of optimal weighted external
difference families to generate optimal weak AMD codes.

This paper is organized as follows. In Section \ref{sec-preliminary}, we
introduces some preliminaries about AMD codes.
In Section \ref{sec-BSWEDF}, we investigate
the relationship between AMD codes and external
difference families. In Section \ref{sec-construction}, we describe several
explicit constructions for bounded standard weighted external difference families,
which are combinatorial equivalents of weak AMD codes.
Conclusion is drawn in Section \ref{sec-conclusion}.

\section{Preliminaries}\label{sec-preliminary}
In this section we describe some notation and definitions about AMD codes.

\begin{itemize}
\item Let $(G,+)$ be an Abelian group of order $n$ with identity $0$;
\item For a positive integer $n$, let $\mathbb{Z}_{n}$ be the residue class ring of integers modulo $n$;
\item For a multi-set $B$ and a positive integer $k$, let $k\boxtimes B$ denote the multi-set, where each
element of $B$ repeated $k$ times;
\item For a subset $B\subseteq G$, ${D(B)}$ denotes the multi-set
$\{a-b\in G: a, b\in B,\,a\ne b\}$;
\item For subsets $B_1,B_2\subseteq G$, $D(B_1,B_2)$ denotes the multi-set
$\{a-b\in G: a\in B_1, b\in B_2\}$;
\item For a multi-set $B$, let $\sharp(a,B)$ denote the number of times that $a$ appears in $B$;
\item For positive integers $k_1,k_2,\dots,k_m$, let $\text{lcm}(k_1,k_2,\dots,k_m)$ denote the least common multiple of
$k_1,k_2,\dots,k_m$.
\end{itemize}

Let $S$ be the source space, i.e., the set of plaintext messages with size $m$, and
$G$ be the encoded message space. An encoding function $E$
maps $s\in S$ to some $g\in G$.
Let $A_s\subseteq G$ denote the set of valid encodings of $s\in S$, where
$A_s\cap A_{s'}=\emptyset$ is required for any $s\ne s'$ so that any message
$g\in A_s$ can be correctly decoded as $D(g)=s$. Denote
$\mathcal{A}\triangleq\{A_s\,:\,s\in S\}$.

\begin{definition}[\cite{PS}]
For given $(S,G,\cA,E)$, let
 \begin{itemize}
  \item {The value $\Delta\in G\backslash \{0\}$ be chosen according to the adversary's strategy $\sigma$;}
  \item {The source message $s\in S$ be chosen uniformly at random by the encoder, i.e., we
  assume equiprobable sources;}
  \item {The message $s$ be encoded into $g\in A_s$ using the encoding function $E$; }
  \item {The adversary wins (a successful tampering) if and only if $g+\Delta\in A_{s'}$ with $s'\ne s$.}
  \end{itemize}
The  probability of successful tampering is denoted by $\rho_\sigma$ for strategy $\sigma$ of the adversary.
The code $(S,G,\mathcal{A},E)$ is called an $(n,m,a,\rho)$ \textit{algebraic manipulation detection code}
(or an $(n,m,a,\rho)$-AMD code for short) under the weak model,
where $a=\sum_{s\in S}|A_s|$ and $\rho$ denotes the maximum probability of successful tampering for all
possible strategies, i.e.,
\begin{equation*}
\rho=\max_{\sigma} \rho_{\sigma}.
\end{equation*}
Specially, if $E$ encodes $s$ to an element of $A_s$ uniformly, i.e.,
$Pr(E(s)=g)=\frac{1}{|A_s|}$ for any $s\in S$ and $g\in A_s$, then we
use $(S,G,\mathcal{A},E_u)$ to distinguish this kind of AMD codes under
the weak model, which were also termed as weak AMD codes in \cite{HP}.
\end{definition}

For weak AMD codes, the following Paterson-Stinson bound was derived in \cite{PS}.
\begin{lemma}[\cite{PS}]\label{lemma_R_optimal}
For any weak $(n,m,a,\rho)$-AMD code, the probability $\rho$ satisfies
\begin{equation*}
\rho\geq \frac{a(m-1)}{m(n-1)}.
\end{equation*}
\end{lemma}

\begin{definition}[\cite{PS}]\label{def_R_op_PS}
A weak AMD code that meets the bound of Lemma \ref{lemma_R_optimal} with
equality is said to be $R$\textit{-optimal} with respect to the bound in
Lemma \ref{lemma_R_optimal}, where $R$ is used to indicate that random
choosing $\Delta$ is an optimal strategy for the adversary.
\end{definition}

\section{Algebraic manipulation detection codes and external difference families}\label{sec-BSWEDF}
In this section, we study the relationship between algebraic manipulation
detection codes and external difference families. Before doing this, we first introduce
some notation and definitions about difference families and their generalizations.

\begin{definition}[\cite{CD}]
Let ${\cB} = \{B_i: 1\le i \le m\}$ be a family of subsets of $G$.
Then ${\cB}$ is called a {\it difference family} (DF) if each nonzero element of $G$
appears exactly $\lambda$ times in the multi-set ${\bigcup}_{1\leq i\leq m}D(B_i)$.
Let $K=(|B_1|,|B_2|,\dots,|B_m|)$. One briefly says that ${\cB}$ is an $(n,K,\lambda)$-DF.
\end{definition}
When $m=1$ the set $B_1$ is also called an $(n,k=|B_1|,\lambda)$ \textit{difference set}.
If $\cB$ forms a partition of $G$,
then ${\cB}$ is called a {\it partitioned difference family} (PDF) \cite{D2009}
and denoted as an $(n,K,\lambda)$-PDF.

\begin{definition}[\cite{PS}]\label{def_nonuniform}
Let ${\cB} = \{B_i: 1\le i \le m\}$ be a family of disjoint subsets of $G$.
Then ${\cB}$ forms an {\it external difference family} (EDF) if each nonzero element of $G$
appears exactly $\lambda$ times in the union of multi-sets $D(B_i,B_j)$ for
$1\leq i\ne j\leq m$,
i.e.,
\begin{equation*}
\bigcup_{1\leq i\ne j\leq m}D(B_i,B_j)=\lambda \boxtimes (G\backslash \{0\}).
\end{equation*}
We briefly denote ${\cB}$ as an $(n,m,K,\lambda)$-EDF, where $K=(|B_1|,|B_2|,\dots,|B_m|)$.
An EDF is \textit{regular} if $|B_1|=|B_2|=\dots=|B_m|=k$, denoted as an $(n,m,k,\lambda)$-EDF,
which is also named as a perfect difference system of sets (refer to \cite{L,FT,FG} for instances).
\end{definition}

\begin{definition}[\cite{PS}]\label{def_nonuniform}
Let ${\cB} = \{B_i: 1\le i \le m\}$ be a family of disjoint subsets of $G$.
Then ${\cB}$ is a {\it bounded external difference family} (BEDF) if each nonzero element of $G$
appears at most $\lambda$ times in the union of multi-sets $D(B_i,B_j)$ for
$1\leq i\ne j\leq m$, i.e.,
\begin{equation*}
\bigcup_{1\leq i\ne j\leq m}D(B_i,B_j)\subseteq\lambda \boxtimes (G\backslash \{0\}).
\end{equation*}
We briefly denote ${\cB}$ as an $(n,m,K,\lambda)$-BEDF, where $K=(|B_1|,|B_2|,\dots,|B_m|)$.
\end{definition}

To construct AMD codes, in \cite{PS}, the following generalizations
of EDF were also introduced.

\begin{definition}[\cite{PS}]\label{def_nonuniform_GSEDF}
Let ${\cB} = \{B_i: 1\le i \le m\}$ be a family of disjoint  subsets of $G$.
${\cB}$ is called an $(n,m;k_1,k_2,\cdots,k_m; $ $\lambda_1,\lambda_2,\cdots,\lambda_m)$-{\it
generalized strong external difference family} (GSEDF) if for any given $1\leq i\leq m$,
each nonzero element of $G$ appears exactly $\lambda_i$  times in the union
of multi-sets $D(B_i,B_j)$ for $1\leq j\ne i\leq m$, i.e.,
\begin{equation}\label{eqn_GSEDF}
\bigcup_{\{j:1\leq j\leq m,\,j\ne i\}}D(B_i,B_j)=\lambda_i \boxtimes (G\backslash \{0\}),
\end{equation}
where $k_i=|B_i|$ for $1\leq i\leq m$.
\end{definition}

\begin{definition}[\cite{PS}]\label{def_nonuniform_GSEDF}
Let ${\cB} = \{B_i: 1\le i \le m\}$ be a family of disjoint  subsets of $G$.
Then ${\cB}$ forms an $(n,m;k_1,k_2,\cdots,k_m;$ $\lambda_1,\lambda_2,\cdots,\lambda_m)$-{\it
bounded generalized strong external difference family} (BGSEDF) if for any given $1\leq i\leq m$,
each nonzero element of $G$ appears at most $\lambda_i$  times in the union
of multi-sets $D(B_i,B_j)$ for $1\leq j\ne i\leq m$, i.e.,
\begin{equation}\label{eqn_BGSEDF}
\bigcup_{\{j:1\leq j\leq m,\,j\ne i\}}D(B_i,B_j)\subseteq\lambda_i \boxtimes (G\backslash \{0\}),
\end{equation}
where $k_i=|B_i|$ for $1\leq i\leq m$.
\end{definition}

\begin{definition}[\cite{PS}]\label{def_nonuniform_PEDF}
Let ${\cB} = \{B_i: 1\le i \le m\}$ be a family of disjoint  subsets of $G$.
Then ${\cB}$ is an $(n,m;c_1,c_2,\cdots,c_l;w_1,$ $w_2,\cdots,w_l;\lambda_1,\lambda_2,
\cdots,\lambda_l)$-{\it partitioned external difference family} (PEDF) if for any given $1\leq t\leq l$,
\begin{equation}\label{eqn_PEDF}
\bigcup_{\{i\,:\,|B_i|=w_t\}}\bigcup_{\{j:1\leq j\leq m,\,j\ne i\}}D(B_i,B_j)=\lambda_t \boxtimes (G\backslash \{0\}),
\end{equation}
where $c_t=|\{i\,:\,|B_i|=w_t,\, 1\leq i\leq m\}|$ for $1\leq t\leq l$.
\end{definition}

To characterize weak AMD codes, we further generalize  external difference families
to  weighted external differences families.

\begin{definition}\label{def_BSWEDF}
Let ${\cB} = \{B_i: 1 \le i \le m\}$ be a family of disjoint subsets of $G$.
Let $K=(k_1,k_2,\dots,k_m)$ with $k_i=|B_i|$ for $1\leq i\leq m$ and
$\widetilde{k}=\text{lcm}(k_1,k_2,\cdots,k_m)$. Define
$\widetilde{\cB}=\{\widetilde{B}_i: B_i\in \cB\}$ as the {\it standard
weighted multi-sets} of $\cB$, where
\begin{equation*}
\widetilde{B}_i\triangleq\frac{\widetilde{k}}{|B_i|}\boxtimes B_i=\frac{\widetilde{k}}{k_i}\boxtimes B_i.
\end{equation*}
Then ${\cB}$ is called an $(n,m,K,a,\lambda)$-{\it bounded standard weighted
external difference family} (BSWEDF) if $\lambda$
is the smallest positive integer such that
\begin{equation*}
\bigcup\limits_{1\leq i\ne j\leq m}D(B_i,\widetilde{B}_j)\subseteq \lambda \boxtimes (G\backslash \{0\}),
\end{equation*}
where $a=\sum_{1\leq i\leq m}k_i$. Furthermore, if
$\cB $ satisfies
\begin{equation*}
\bigcup\limits_{1\leq i\ne j\leq m}D(B_i,\widetilde{B}_j)=\lambda \boxtimes (G\backslash \{0\}),
\end{equation*}
then it is named as a {\it standard weighted external difference family}, also denoted as
an $(n,m,K,a,\lambda)$-SWEDF for short.
\end{definition}

For BSWEDFs and SWEDFs, we have the following facts on their parameters.
\begin{lemma}\label{lemma_bound_BSWEDF}
Let $\cB$ be an $(n,m,K,a,\lambda)$-BSWEDF. Then we have
\begin{equation}\label{eqn_Bound_lambda}
\lambda\geq \left\lceil\frac{\widetilde{k}a(m-1)}{n-1}\right\rceil.
\end{equation}
Specially, if $\cB$ is an $(n,m,K,a,\lambda)$-SWEDF, then
$(n-1)\mid (\widetilde{k}a(m-1))$ and
\begin{equation*}
\lambda=\frac{\widetilde{k}a(m-1)}{n-1}.
\end{equation*}
\end{lemma}
\begin{proof}
Let $\cB=\{B_i\,:\,1\leq i\leq m\}$.
The fact
\begin{equation*}
\bigcup\limits_{1\leq i\ne j\leq m}D(B_i,\widetilde{B}_j)=\bigcup\limits_{1\leq i\ne j\leq m}\bigcup_{b\in B_i}D(\{b\},\widetilde{B}_j)
\end{equation*}
means that
\begin{equation}\label{eqn_total_dif}
\left|\bigcup\limits_{1\leq i\ne j\leq m}D(B_i,\widetilde{B}_j)\right|=\sum\limits_{1\leq i\leq m}\sum\limits_{1\leq j\leq m\atop j\ne i}\sum_{b\in B_i}|D(\{b\},\widetilde{B}_j)|
=\sum\limits_{1\leq i\leq m}\sum\limits_{1\leq j\leq m\atop j\ne i}\sum_{b\in B_i}\widetilde{k}=\widetilde{k}a(m-1).
\end{equation}
Thus, we have
$\lambda \geq \lceil\frac{\widetilde{k}a(m-1)}{n-1}\rceil$.

Similarly, for the case of SWEDFs,
by Definition \ref{def_BSWEDF} and \eqref{eqn_total_dif}, we have $\lambda(n-1)=\widetilde{k}a(m-1)$,
i.e., $\lambda=\frac{\widetilde{k}a(m-1)}{n-1}$, which also means $(n-1)\mid(\widetilde{k}a(m-1))$.
\end{proof}

\begin{definition}
An $(n,m,K,a,\lambda)$-BSWEDF is said to be \textit{optimal}
if $\lambda$ takes the smallest possible value for
given $n$, $m$, and $K$.

\end{definition}

Specially, an $(n,m,K,a,\lambda)$-BSWEDF
is optimal if $\lambda$ achieves the lower bound given by
\eqref{eqn_Bound_lambda} with equality, i.e.,
$\lambda=\lceil\frac{\widetilde{k}a(m-1)}{n-1}\rceil$.

For $\Delta\in G\backslash \{0\}$, let $\rho_\Delta$ denote
the probability  that the adversary wins by modifying $g\in A_s$ into $g+\Delta\in A_{s'}$ for some $s'\neq s$. Thus, we have
$\rho=\max\{\rho_\Delta:\Delta\in G\backslash \{0\}\}$.
\begin{theorem}\label{theorem_AMD_BSWEDF}
There exists a weak $(n,m,a,\rho)$-AMD code $(S,G,\cA,E_u)$ if and only
if there exists
an $(n,m,K,a,\lambda)$-BSWEDF, where $|G|=n$, $a=\sum_{1\leq i\leq m}|A_{s_i}|$,
$K=(|A_{s_1}|,|A_{s_2}|,\cdots, |A_{s_m}|)$, $s_i\in S$, and
$\rho=\frac{\lambda}{\widetilde{k}m}$.
\end{theorem}
\begin{proof}
If $(S,G,\cA,E_u)$ is a weak $(n,m,a,\rho)$-AMD code, then for any
$\Delta\in G\backslash\{0\}$, we have
\begin{equation*}
\rho_\Delta\leq \rho=\frac{\lambda}{\widetilde{k}m},
\end{equation*}
that is,
\begin{equation}\label{eqn_rho_delta_stand}
\begin{split}
\frac{\lambda}{\widetilde{k}m}\geq \rho_\Delta
=&\sum_{s\in S}Pr(s)\sum_{g\in A_s}Pr(E_u(s)=g)\left(\sum_{s'\ne s,s'\in S}Pr(g+\Delta\in A_{s'})\right)\\
=&\sum_{s\in S}\frac{1}{m}\sum_{g\in A_s}\frac{1}{|A_s|}\left(\sum_{s'\ne s,s'\in S}Pr(g+\Delta\in A_{s'})\right)\\
=&\sum_{s\in S}\frac{1}{m}\frac{1}{|A_s|}\left(\sum_{s'\ne s,s'\in S}\sum_{g\in A_s}Pr(g+\Delta\in A_{s'})\right),
\end{split}
\end{equation}
where the second equality holds by the fact that $E_u$ encodes $s$ to elements of $A_s$ with uniform probability.
Note that for given $\Delta$, $s$, $g\in A_s$ and $s'\ne s$,
\begin{equation*}
Pr(g+\Delta\in A_{s'})=
\begin{cases}
1,\,\,&\Delta\in D(A_{s'},\{g\}),\\
0,\,\,&\Delta\not\in D(A_{s'},\{g\}).\\
\end{cases}
\end{equation*}

Thus, Inequality \eqref{eqn_rho_delta_stand} implies that
\begin{equation}\label{eqn_rho_delta}
\begin{split}
\frac{\lambda}{m}\geq\widetilde{k}\rho_\Delta=&\sum_{s\in S}\frac{1}{m}\frac{\widetilde{k}}{|A_s|}\left(\sum_{s'\ne s,s'\in S}\sum_{g\in A_s}Pr(g+\Delta\in A_{s'})\right)\\
=&\sum_{s\in S}\frac{1}{m}\frac{\widetilde{k}}{|A_s|}\left(\sum_{s'\ne s,s'\in S}\sharp\left(\Delta, D(A_{s'},A_s)\right)\right)\\
=&\sum_{s\in S}\frac{1}{m}\left(\sum_{s'\ne s,s'\in S}\frac{\widetilde{k}}{|A_s|}\sharp\left(\Delta, D(A_{s'},A_s)\right)\right)\\
=&\sum_{s\in S}\frac{1}{m}\left(\sum_{s'\ne s,s'\in S}\sharp\left(\Delta, D(A_{s'},\widetilde{A}_s)\right)\right)\\
=&\frac{1}{m}\sharp\left(\Delta, \bigcup_{s,s'\in S,\atop s'\ne s} D(A_{s'},\widetilde{A}_s)\right),\\
\end{split}
\end{equation}
where $\sharp(\Delta, B)$ denotes the number of times that $\Delta$ appears in the multi-set $B$.
This means that any $\Delta\in G\backslash \{0\}$ appears at most
$\lambda$ times in the multi-set
$\bigcup_{s,s'\in S,\atop s'\ne s} D(A_{s'},\widetilde{A}_s)$,
i.e., $$\bigcup_{s,s'\in S,\atop s'\ne s} D(A_{s'},\widetilde{A}_s)\subseteq  \lambda\boxtimes(G\backslash\{0\}).$$
Note that $\rho=\max\{\rho_\Delta\,:\,\Delta\in G\backslash\{0\}\}$ means
there exists at least one $\Delta\in G\backslash \{0\}$ such that the equality
in \eqref{eqn_rho_delta} holds.
Then $\{A_s\,:\,s\in S\}$ forms an $(n,m,(|A_{s_1}|,|A_{s_2}|,\cdots,|A_{s_m}|),a,\lambda)$-BSWEDF
by Definition \ref{def_BSWEDF}.

Conversely, suppose that there exists an $(n,m,K,a,\lambda)$-BSWEDF $\cB=\{B_i\,:\,1\leq i\leq m\}$ over $G$. Let
$S=\{s_i\,:\,1\leq i\leq m\}$ and $A_{s_i}=B_i$ for $1\leq i\leq m$. Then we can define
a weak AMD code, where $E_u(s_i)=g\in B_i$ with equiprobability. For any $\Delta\in G\backslash\{0\}$,
similarly as \eqref{eqn_rho_delta_stand}, we have
\begin{equation*}
\begin{split}
\rho_\Delta=&\sum_{s\in S}\frac{1}{m}\frac{1}{|A_s|}\left(\sum_{s'\ne s,s'\in S}\sum_{g\in A_s}Pr(g+\Delta\in A_{s'})\right)\\
=&\sum_{1\leq i\leq m}\frac{1}{m}\frac{1}{|B_i|}\left(\sum_{1\leq j\leq m\atop j\ne i}\sharp(\Delta,D(B_j,B_i))\right)\\
=&\sum_{1\leq i\leq m}\frac{1}{\widetilde{k}m}\left(\sum_{1\leq j\leq m\atop j\ne i}\sharp(\Delta,D(B_j,\widetilde{B}_i))\right)\\
=&\frac{1}{\widetilde{k}m}\left(\sum_{1\leq j\ne i\leq m}\sharp(\Delta,D(B_j,\widetilde{B}_i))\right)\\
\leq &\frac{\lambda}{\widetilde{k}m},
\end{split}
\end{equation*}
where the last inequality holds by the fact that $\cB$ is an $(n,m,K,a,\lambda)$-BSWEDF.
According to Definition \ref{def_BSWEDF}, the equality is achieved for at least one
$\Delta\in G\backslash\{0\}$ in the preceding inequality.
Thus, the weak $(n,m,a,\rho)$-AMD code
defined based on the BSWEDF $\cB$ satisfies $$\rho=\max\{\rho_\Delta\,:\, \Delta\in
G\backslash\{0\}\}= \frac{\lambda}{\widetilde{k}m},$$ which completes the proof.
\end{proof}


When we consider the optimality of BSWEDF, the size-distribution $K=(k_1, k_2,\dots,k_m)$
is given. However, the $R$-optimality of weak AMD codes only relates with $a=\sum_{1\leq i\leq m}k_i$
as defined in \cite{PS} but
disregards the exact size-distribution $K$ of $\cA$.
There may exist several BSWEDFs with different $K$ which correspond to weak AMD codes
with exactly the same parameter $a$. Thus, although the BSWEDF gives a characterization of the weak AMD code,
in general, the optimal BSWEDF for a given $K$ does not necessarily correspond to an $R$-optimal
weak AMD code for a given $a$.

\begin{definition}\label{def_strong_optimal}
For given $n$, $m$ and $a$, an $(n,m,K,a,\lambda)$-BSWEDF is said to be \textit{strongly optimal}
if $\frac{\lambda}{\widetilde{k}m}=\rho_{(n,m,a)}$, where
\begin{equation}\label{eqn_rh_nma}
\rho_{(n,m,a)}=\min_{K'}\left\{\frac{\lambda'}{\widetilde{k'}m}\,:\,\exists\, (n,m,K',a,\lambda')\text{-BSWEDF}\, s.t.\, \sum_{1\leq i\leq m}k'_i=a\right\}.
\end{equation}
\end{definition}

By Theorem \ref{theorem_AMD_BSWEDF} and Lemma \ref{lemma_bound_BSWEDF}, we have

\begin{corollary}\label{corollary_improved_bound}
For any weak $(n,m,a,\rho)$-AMD code $(S,G,\mathcal{A},E_u)$, we have
\begin{equation*}
\rho\geq \rho_{(n,m,a)}\geq \min_{K}\left\{\left\lceil\frac{\widetilde{k}a(m-1)}{n-1}
\right\rceil\frac{1}{\widetilde{k}m}\,:\,\sum_{1\leq i\leq m}k_i=a\right\},
\end{equation*}
where $|A_i|=k_i$ for any $A_i \in \cA$.
\end{corollary}
\begin{proof}
Let $(S,G,\mathcal{A},E_u)$ be a weak $(n,m,a,\rho)$-AMD code.
By Theorem \ref{theorem_AMD_BSWEDF}, there exists an $(n,m,K,a,\lambda)$-BSWEDF
with $\lambda= \widetilde{k}m\rho$. Then by Lemma \ref{lemma_bound_BSWEDF} and
\eqref{eqn_rh_nma},
$$\rho=\frac{\lambda}{\widetilde{k}m}\geq \rho_{(n,m,a)}\geq\min_{K}\left\{\left\lceil\frac{\widetilde{k}a(m-1)}{n-1}\right\rceil
\frac{1}{\widetilde{k}m}\,:\,\sum_{1\leq i\leq m}k_i=a\right\}.$$
\end{proof}

\begin{definition}
A weak AMD code with $\rho=\rho_{(n,m,a)}$ is said to be $R$\textit{-optimal} with respect to the bound in
Corollary \ref{corollary_improved_bound}.
\end{definition}

When $(n-1)\mid (\widetilde{k}a(m-1))$,
the bound in Corollary \ref{corollary_improved_bound} is exactly the same as the
one given in Lemma \ref{lemma_R_optimal}.
However, when $(n-1) \nmid (\widetilde{k}a(m-1))$, our bound
in Corollary \ref{corollary_improved_bound} can improve
the known one in Lemma \ref{lemma_R_optimal}.
The following is an easy example.

\begin{corollary}
For any weak $(n,m,a,\rho)$-AMD code $(S,G,\mathcal{A},E_u)$, if $n-1$ is a prime and $a<n-1$, then we have
\begin{equation*}
\rho\geq \min_{K}\left\{\left\lceil\frac{\widetilde{k}a(m-1)}{n-1}\right\rceil
\frac{1}{\widetilde{k}m}\,:\,\sum_{1\leq i\leq m}k_i=a\right\}>\frac{a(m-1)}{m(n-1)}.
\end{equation*}
\end{corollary}
\begin{proof}
The lemma follows from the facts that $k_i\leq a<n-1$ for $1\leq i\leq m$, $m\leq a<n-1$, and $n-1$ is a prime.
In this case, $(n-1)\nmid(\widetilde{k}a(m-1))$.
\end{proof}

A more concrete example is listed below.

\begin{example}
Let $n=10$, $m=3$, and $a=5$. Let $B=\{\{5\},\{2\},\{0,4,6\}\}$ be a family of disjoint subsets of $\Z_{10}$,
which corresponding to a weak $(10,3,5,\rho)$-AMD code, where $\rho=\frac{1}{3}\cdot\frac{1}{1}\cdot{1}+\frac{1}{3}\cdot\frac{1}{1}\cdot{0}+\frac{1}{3}\cdot\frac{1}{3}\cdot{1}=\frac{4}{9}$.
According to Lemma \ref {lemma_R_optimal} and definition \ref{def_R_op_PS}, this is not an $R$-optimal weak AMD code.
However, $R$-optimality should mean that random choosing $\Delta$ is an optimal strategy for the adversary.
Clearly, according to Corollary \ref{corollary_improved_bound}, the parameter $\rho$ cannot be smaller then
\begin{equation*}
\begin{split}
&\min_{K}\left\{\left\lceil\frac{\widetilde{k}5(3-1)}{10-1}
\right\rceil\frac{1}{3\widetilde{k}}\,:\,\sum_{1\leq i\leq 3}k_i=5\right\}\\
=&\min\left\{\left\lceil\frac{\lcm (1,1,3)\cdot5\cdot2}{9}\right\rceil\frac{1}{3\lcm(1,1,3)},\left\lceil\frac{\lcm(1,2,2)\cdot5\cdot2}{9}\right\rceil\frac{1}{3\lcm(1,2,2)}\right\}\\
=&\min\left\{\frac{4}{9},\frac{1}{2}\right\}=\frac{4}{9}.\\
\end{split}
\end{equation*}

Therefore, this example should be an $R$-optimal weak $(10,3,5,\rho)$-AMD code.
This trouble is due to the fact that the known bound in Lemma \ref {lemma_R_optimal} is not always tight.
\end{example}

Relationships between optimal weak AMD codes and optimal BSWEDFs are described below.

\begin{corollary}\label{corollary_char}
Let $n$ and $m$ be positive integers.
\begin{itemize}
\item[(I)] {For given $K=(k_1,k_2,\dots,k_m)$, let $\rho_{(n,m,K)}$ denote the
the smallest possible $\rho$ for weak $(n,m,\sum_{1\leq i\leq m}k_i,\rho)$-AMD codes.
Then a  weak $(n,m,a,\rho)$-AMD
code $(S,G,\cA,E_u)$ has the smallest $\rho$, i.e., $\rho=\rho_{(n,m,K)}$ if and only if its corresponding BSWEDF with parameters
$(n,m,K,a,\lambda=\widetilde{k}m\rho)$ is optimal, where
$S=\{s_i\,:\,1\leq i\leq m\}$, $\cA=\{A_{s_i}\,:\,1\leq i\leq m\}$, $k_i=|A_{s_i}|$ for $1\leq i\leq m$, $K=(k_1,k_2,\dots,k_m)$, and $a=\sum_{1\leq i\leq m}k_i$.}

\item[(II)] {For given $a$, there exists an $R$-optimal weak $(n,m,a,\rho)$-AMD
code $(S,G,\cA,E_u)$ with respect to the bound in Corollary \ref{corollary_improved_bound}
if and only if there exists
a strongly optimal $(n,m,K,a,\lambda)$-BSWEDF, where $|G|=n$, $a=\sum_{s\in S}|A_s|$,
$\rho=\rho_{(n,m,a)}$,
and $\lambda=\widetilde{k}m\rho_{(n,m,a)}$.}

\item[(III)]{There exists an $R$-optimal weak $(n,m,a,\rho)$-AMD code $(S,G,\cA,E_u)$ with respect to
the bound in Lemma \ref{lemma_R_optimal} if and only
if there exists
an $(n,m,K,a,\lambda)$-SWEDF,  where $\rho=\frac{a(m-1)}{m(n-1)}$,
and $\lambda=\frac{\widetilde{k}a(m-1)}{n-1}$.}
\end{itemize}
\end{corollary}
\begin{proof}
By Theorem \ref{theorem_AMD_BSWEDF}, for given $n$, $m$, $K$ (or $a$, resp.), a
weak AMD code with the smallest $\rho$ is equivalent to a
BSWEDF with the smallest $\lambda$, i.e., an optimal (or strongly optimal, resp.) BSWEDF.
The third part of the result follows directly
from Theorem \ref{theorem_AMD_BSWEDF} and Lemma \ref{lemma_bound_BSWEDF}.
\end{proof}

\begin{example}
Let $n=10$, $m=3$, and $a=5$. Let $\cB^{(1)}=\{B^{(1)}_1=\{5\},B^{(1)}_2=\{4,6\},B^{(1)}_3=\{2,8\}\}$ and
$\cB^{(2)}=\{B^{(2)}_1=\{5\},B^{(2)}_2=\{2\},B^{(2)}_3=\{0,4,6\}\}$ be two families of disjoint subsets
of $\Z_{10}$. It is easy to verify that
\begin{equation*}
\bigcup_{1\leq i\leq 3}D\left(B^{(1)}_i,\widetilde{B}^{(1)}_j\right)\subseteq 3\boxtimes (\Z_{10}\backslash\{0\})
\end{equation*}
and
\begin{equation*}
\bigcup_{1\leq i\leq 3}D\left(B^{(2)}_i,\widetilde{B}^{(2)}_j\right)\subseteq 4\boxtimes (\Z_{10}\backslash\{0\}).
\end{equation*}
According to Lemma \ref{lemma_bound_BSWEDF}, $\cB^{(1)}$ is an optimal $(10,3,(1,2,2),5,3)$-BSWEDF and
$\cB^{(2)}$ is an optimal $(10,3,(1,1,3),5,4)$-BSWEDF. By Corollary \ref{corollary_improved_bound},
$$\rho_{(10,3,5)}\geq \min_{K}\left\{\left\lceil\frac{\widetilde{k}5(3-1)}{10-1}
\right\rceil\frac{1}{3\widetilde{k}}\,:\,\sum_{1\leq i\leq 3}k_i=5\right\}=\frac{4}{9}.$$
Thus, by Definition \ref{def_strong_optimal}, $\cB^{(2)}$ is in fact not only an optimal, but a strongly optimal BSWEDF.
By Corollary \ref{corollary_char}. (II), we can obtain a corresponding  $R$-optimal weak AMD code with respect
to the bound in Corollary \ref{corollary_improved_bound} from $\cB^{(2)}$.
\end{example}

Although the weak $(n,m,a,\rho_{(n,m,K)}=\frac{\lambda}{\widetilde{k}m})$-AMD
code $(S,G,\cA,E_u)$ based on an optimal $(n,m,K,a,\lambda)$-BSWEDF
may sometimes not correspond to an $R$-optimal weak AMD code with parameters
$(n,m,a,\rho_{(n,m,a)})$, the difference $\rho_{(n,m,K)}-\rho_{(n,m,a)}$ is not big.

\begin{lemma}
Let $a=\sum_{A\in \cA}|A|=\sum_{1\leq i\leq m}k_i$.
Let $(S,G,\cA,E_u)$ be the weak $(n,m,a,\rho=\frac{\lambda}{\widetilde{k}m})$-AMD
code based on an optimal
$(n,m,K,a,\lambda)$-BSWEDF with $\lambda=\lceil\frac{\widetilde{k}a(m-1)}{n-1}\rceil$,
and let $(S,G,\cA',E_u)$ be the $R$-optimal weak $(n,m,a,\rho_{(n,m,a)})$-AMD
code with respect to the bound in
Corollary \ref{corollary_improved_bound}. Then we have
\begin{equation*}
0 \leq\rho_{(n,m,K)}-\rho_{(n,m,a)}\leq \frac{1}{\widetilde{k}m}.
\end{equation*}
\end{lemma}
\begin{proof}
The lemma follows directly from the fact that
$$0 \leq\rho_{(n,m,K)}-\rho_{(n,m,a)}= \left\lceil\frac{\widetilde{k}a(m-1)}{n-1}\right\rceil\frac{1}{\widetilde{k}m}-\rho_{(n,m,a)}\leq \left\lceil\frac{\widetilde{k}a(m-1)}{n-1}\right\rceil\frac{1}{\widetilde{k}m}-\frac{a(m-1)}{m(n-1)}\leq \frac{1}{\widetilde{k}m}.$$

\end{proof}

In \cite{HP}, Huczynska and Paterson characterized  $R$-optimal AMD codes
$(S,G,\cA,E_u)$ by reciprocally-weighted external difference families,
which can be defined as follows.

\begin{definition}[\cite{HP}]\label{def_RWEDF}
Let ${\cB} = \{B_i: 1 \le i \le m\}$ be a family of subsets of $G$.
Let $K=(k_1,k_2,\cdots,k_m)$ with $k_i=|B_i|$ for $1\leq i\leq m$ and
$\widetilde{k}=\text{lcm}(k_1,k_2,\cdots,k_m)$. Then  $\cB$
is said to be an $(n,m,(k_1,k_2,\cdots,k_m),d)$ \textit{reciprocally-weighted external
difference family} (RWEDF) if
\begin{equation*}
d=\sum_{1\leq i\leq m}\frac{N_i(\delta)}{k_i} \text{ for each }\delta\in G\backslash \{0\},
\end{equation*}
where $$N_i(\delta)\triangleq \left|\left\{(b_i,b_j):b_i\in B_i,\,\,
b_j\in\bigcup_{1\leq t\ne i\leq m}B_t, \text{ and }b_j-b_i=\delta\right\}\right|.$$
\end{definition}

\begin{theorem}[\cite{HP}]\label{theorem_AMD_RWEDF}
A weak $(n,m,a,\rho)$-AMD code $(S,G,\cA,E_u)$ is $R$-optimal with respect to
the bound in Lemma \ref{lemma_R_optimal} if and only
if there exists
an $(n,m,K,a,d)$-RWEDF,  where $\rho=\frac{a(m-1)}{m(n-1)}$,
and $d=\frac{a(m-1)}{n-1}$.
\end{theorem}

Clearly, $N_i(\delta)=\sharp\left(\delta, \bigcup_{1\leq j\leq m\atop j\ne i} D(B_j,B_i)\right)$
for $1\leq i\leq m$, and
by Theorem \ref{theorem_AMD_RWEDF} and Corollary \ref{corollary_char}
or Definitions \ref{def_BSWEDF} and \ref{def_RWEDF},
we know that an $(n,m,K,a,d)$-RWEDF is essentially the same as an $(n,m,K,a,\lambda)$-SWEDF, where
$d=\frac{\lambda}{\widetilde{k}}$. Therefore, Theorem \ref{theorem_AMD_BSWEDF} and Corollary
\ref{corollary_char} provide more
combinatorial characterizations for various weak AMD codes $(S,G,\cA,E_u)$.
These results can be viewed as a generalization of Theorem \ref{theorem_AMD_RWEDF}.
As a byproduct, we have the following property
for an $(n,m,K,a,d)$-RWEDF directly from Lemma \ref{lemma_bound_BSWEDF}
and Corollary \ref{corollary_char}. (III).
\begin{corollary}
A necessary condition for the existence of an $(n,m,K,a,d)$-RWEDF, or equivalently
an $R$-optimal weak $(n,m,a,\rho)$-AMD code $(S,G,\cA,E_u)$ with respect to Lemma \ref{lemma_R_optimal},
is $(n-1)\mid (\widetilde{k}a(m-1))$, where $K=(k_1=|A_{s_1}|,k_2=|A_{s_2}|,\cdots,k_m=|A_{s_m}|)$ and
$\widetilde{k}=\lcm(k_1,k_2,\cdots,k_m)$.
\end{corollary}

In Figure \ref{figure AMD}, we summarize the relationships between weak AMD codes and
BSWEDFs, where SO-BSWEDF, O-BSWEDF, and OW-AMD-code denote strongly optimal
BSWEDF, optimal BSWEDF, and $R$-optimal weak AMD-code, respectively.

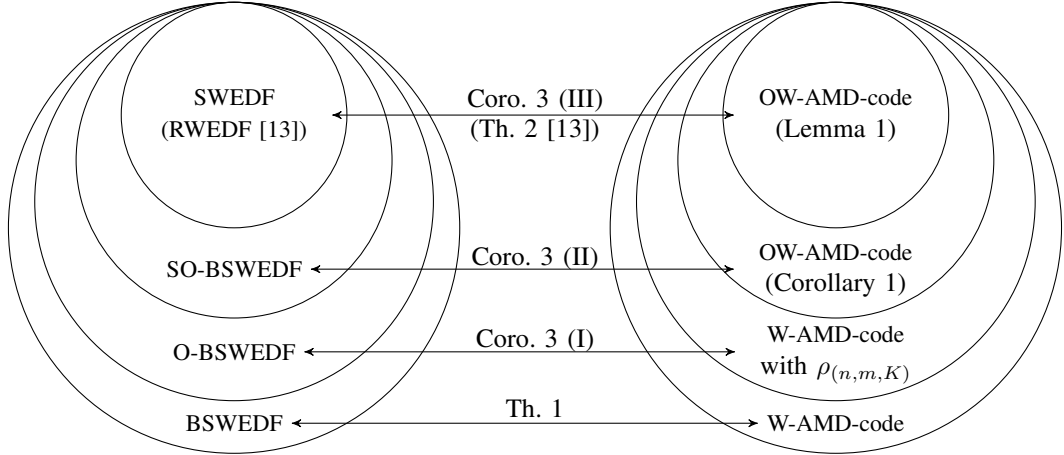
\begin{figure}[!t]
\centering
\begin{tikzpicture}
\tikzset{venn circle/.style={draw,circle,opacity=1}}
\node [venn circle, minimum width=6cm] (A) at (0,0) {};
\node [venn circle, minimum width=5.3cm] (B) at (0,0.35) {};
\node [venn circle, minimum width=4.2cm ] (C) at (0,0.9) {};
\node [venn circle, minimum width=3cm ] (D) at (0,1.5) {};

\node [] (SWEDF) at (0,1.5cm) {\begin{tabular}{c}
 \small SWEDF\\
\small (RWEDF \cite{HP})\\
  \end{tabular}};
\node [] (BSWEDF) at (0,-2.6cm) {\small BSWEDF};
\node [] (O-BSWEDF) at (0,-1.65cm) {\small O-BSWEDF};
\node [] (SO-BSWEDF) at (0,-0.55cm) {\small SO-BSWEDF};

\node [venn circle, minimum width=6cm] (A1) at (8,0) {};
\node [venn circle, minimum width=5.3cm] (B1) at (8,0.35) {};
\node [venn circle, minimum width=4.2cm ] (C1) at (8,0.9) {};
\node [venn circle, minimum width=3cm ] (D1) at (8,1.5) {};

\node [] (ROWAMD) at (8,1.5cm) {\begin{tabular}{c}
 \small OW-AMD-code\\
(Lemma \ref{lemma_R_optimal})\\
  \end{tabular}};
\node [] (AMD) at (8,-2.6cm) {\small W-AMD-code};
\node [] (KOAMD) at (8,-1.65cm) {\begin{tabular}{c}
 \small W-AMD-code\\
with $ \rho_{(n,m,K)}$\\
  \end{tabular}};
\node [] (OWAMD) at (8,-0.55cm) {\begin{tabular}{c}
 \small OW-AMD-code\\
(Corollary \ref{corollary_improved_bound})\\
  \end{tabular}};

\draw[black,<->] (ROWAMD) -- (SWEDF) node at (4,1.7cm) {Coro. \ref{corollary_char} (III) };
\node [] (REMARK) at (4,1.3cm) {(Th. 2 \cite{HP})};
\draw[black,<->] (KOAMD) -- (O-BSWEDF) node at (4,-0.4cm) {Coro. \ref{corollary_char} (II)};
\draw[black,<->] (OWAMD) -- (SO-BSWEDF) node at (4,-1.5cm) {Coro. \ref{corollary_char} (I)};
\draw[black,<->] (AMD) -- (BSWEDF) node at (4,-2.4cm) {Th. \ref{theorem_AMD_BSWEDF}};
\end{tikzpicture}
\caption{The relationships between AMD codes and BSWEDFs}
\label{figure AMD}
\end{figure}

\subsection{Among EDFs, SEDFs, PEDFs, SWEDFs, and BSWEDFs}
In general, an EDF is not necessarily an SWEDF. However, in the following cases,
an EDF is always an SWEDF. First of all, we consider the regular case.
\begin{lemma}
A regular $(n,m,k,\lambda)$-EDF forms an $(n,m,K=(k,k,\dots,k),a=mk,\lambda)$-SWEDF.
\end{lemma}
The lemma follows directly from the definitions of EDF and SWEDF.

For the case of GSEDFs we have the following result.
\begin{lemma}
If $\{B_i\,:\, 1\leq i\leq m\}$ is an $(n,m;k_1,k_2,\cdots,k_m;
\lambda_1,\lambda_2,\cdots,\lambda_m)$-GSEDF, then $\{B_i\,:\, 1\leq i\leq m\}$
is an $(n,m,(k_1,k_2,\cdots,k_m),a,\lambda)$-SWEDF, where $\lambda=\sum_{1\leq i\leq m}
\frac{\lambda_i\widetilde{k}}{k_i}$.
\end{lemma}
\begin{proof}
Let $\{B_i\,:\, 1\leq i\leq m\}$ be an $(n,m;k_1,k_2,\cdots,k_m;
\lambda_1,\lambda_2,\cdots,\lambda_m)$-GSEDF, by \eqref{eqn_GSEDF},
\begin{equation*}
\bigcup_{1\leq j\leq m,\,j\ne i}D(B_i,B_j)=\lambda_i \boxtimes (G\backslash \{0\}),
\end{equation*}
which means $$\bigcup_{1\leq j\leq m,\,j\ne i}D(B_j,\widetilde{B}_i)=\frac{\lambda_i\widetilde{k}}{k_i} \boxtimes (G\backslash \{0\}).
$$
Thus, we have $$\bigcup_{1\leq i\leq m}\bigcup_{1\leq j\leq m,\,j\ne i}D(B_j,\widetilde{B}_i)
=\left(\sum_{1\leq i\leq m}\lambda_i\frac{\widetilde{k}}{k_i}\right)\boxtimes (G\backslash \{0\})=\lambda\boxtimes (G\backslash \{0\}),
$$ i.e., $\{B_i\,:\, 1\leq i\leq m\}$
is an $(n,m,(k_1,k_2,\cdots,k_m),a,\lambda)$-SWEDF with $\lambda=\sum_{1\leq i\leq m}
\frac{\lambda_i\widetilde{k}}{k_i}$.
\end{proof}

Similarly, the relationship between PEDFs and SWEDFs can be given by the
following lemma.

\begin{lemma}
If $\{B_i\,:\, 1\leq i\leq m\}$ is an $(n,m;c_1,c_2,\cdots,c_l;w_1,w_2,\cdots,w_l;
\lambda_1,\lambda_2,\cdots,\lambda_l)$-PEDF, then $\{B_i\,:\, 1\leq i\leq m\}$
is an $(n,m,K=(|B_1|,|B_2|,\cdots,|B_m|),a,\lambda)$-SWEDF, where
$\widetilde{k}=\text{lcm}(w_1,w_2,\cdots, w_l)$ and $\lambda=\sum_{1\leq t\leq l}
\frac{\lambda_t\widetilde{k}}{w_t}$.
\end{lemma}
\begin{proof}
Since $\{B_i\,:\, 1\leq i\leq m\}$ is an $(n,m;c_1,c_2,\cdots,c_l;w_1,w_2,\cdots,w_l;
\lambda_1,\lambda_2,\cdots,\lambda_l)$-PEDF, by \eqref{eqn_PEDF},
\begin{equation*}
\bigcup_{\{i\,:\,|B_i|=w_t\}}\bigcup_{1\leq j\leq m,\,j\ne i}D(B_i,B_j)=\lambda_t \boxtimes (G\backslash \{0\})
\end{equation*}
for $1\leq t\leq l$. By Definition \ref{def_nonuniform_PEDF}, $|B_i|\in \{w_j\,:\,1\leq j\leq l\}$
for $1\leq i\leq m$.
Thus, for $K=(|B_1|,|B_2|,\cdots, |B_m|)$, we have
$\widetilde{k}=\text{lcm}(|B_1|,|B_2|,\cdots, |B_m|)=\text{lcm}(w_1,w_2,\cdots, w_l).$
Thus, we have $$\bigcup_{1\leq t\leq l}\bigcup_{\{i\,:\,|B_i|=w_t\}}\bigcup_{1\leq j\leq m,\,j\ne i}D(B_j,\widetilde{B}_i)=\left(\sum_{1\leq t\leq l}\lambda_t\frac{\widetilde{k}}{w_t}\right)\boxtimes (G\backslash \{0\})=\lambda \boxtimes (G\backslash \{0\}),
$$ i.e., $\{B_i\,:\, 1\leq i\leq m\}$
is an $(n,m,K=(|B_1|,|B_2|,\cdots,|B_m|),a,\lambda)$-SWEDF, where $\lambda=\sum_{1\leq t\leq l}
\frac{\lambda_t\widetilde{k}}{w_t}$.
\end{proof}

In what follows, we recall an example of SWEDF which is not an EDF, or an GSEDF, or a PEDF.

\begin{example}[\cite{PS}]
Let $G=(\Z_{10},+)$ and $\cB=\{B_1=\{0\},B_2=\{5\},B_3=\{2,3\},B_4=\{6,4\}\}$. Then
$\widetilde{B}_1=\{0,0\},\widetilde{B}_2=\{5,5\},\widetilde{B}_3=\{2,3\},
\widetilde{B}_4=\{6,4\}$. It is easy to check
$$\bigcup_{1\leq i\leq 4}\bigcup_{1\leq j\leq 4,\, j\ne i}D(B_i,\widetilde{B}_j)
=4\boxtimes (G\backslash\{0\}),$$
$$\bigcup_{1\leq i\leq 4}\bigcup_{1\leq j\leq 4,\, j\ne i}D(B_i,B_j)
\ne \lambda \boxtimes(G\backslash\{0\}),$$
$$\bigcup_{2\leq j\leq 4}D(B_1,B_j)
=\{5,8,7,4,6\}\ne \lambda \boxtimes (G\backslash\{0\}),$$
and
$$\bigcup_{3\leq i\leq 4}\bigcup_{1\leq j\leq 4,\, j\ne i}D(B_i,B_j)
\ne \lambda\boxtimes (G\backslash\{0\}),$$
for any positive integer $\lambda$.
Thus, $\cB$ is an SWEDF which
does not form an EDF, or a GSEDF, or a PEDF.
\end{example}

Similarly, a BEDF is not necessarily a BSWEDF in general and we have the following
relationship between BEDFs and BSWEDFs.

\begin{lemma}
The regular $(n,k,\lambda)$-BEDF forms an $(n,m,K=(k,k,\dots,k),a=mk,\lambda_1)$-BSWEDF, where
$\lambda_1\leq \lambda$.
\end{lemma}

\begin{lemma}
If $\cB=\{B_i\,:\, 1\leq i\leq m\}$ is an $(n,m;k_1,k_2,\cdots,k_m;
\lambda_1,\lambda_2,\cdots,\lambda_m)$-BGSEDF, then $\cB$
is an $(n,m,(k_1,k_2,\cdots,$ $k_m),a=\sum_{1\leq i\leq m}k_i,\lambda)$-BSWEDF, where $\lambda\leq\sum_{1\leq i\leq m}
\frac{\lambda_i\widetilde{k}}{k_i}$.
\end{lemma}
\begin{proof}
Since $\cB=\{B_i\,:\, 1\leq i\leq m\}$ is an $(n,m;k_1,k_2,\cdots,k_m;
\lambda_1,\lambda_2,\cdots,\lambda_m)$-BGSEDF, by \eqref{eqn_BGSEDF},
\begin{equation*}
\bigcup_{1\leq j\leq m,\,j\ne i}D(B_i,B_j)\subseteq \lambda_i \boxtimes (G\backslash \{0\}),
\end{equation*}
which means
\begin{equation}\label{eqn_BGSEDF_BSWEDF}
\bigcup_{1\leq j\leq m,\,j\ne i}D(B_j,\widetilde{B}_i)\subseteq \lambda_i\frac{\widetilde{k}}{k_i} \boxtimes (G\backslash \{0\}).
\end{equation}
Let $\lambda$ be the smallest positive integer such that
\begin{equation*}
\bigcup_{1\leq i\leq m}\bigcup_{1\leq j\leq m,\,j\ne i}D(B_j,\widetilde{B}_i)\subseteq \lambda \boxtimes (G\backslash \{0\}).
\end{equation*}
Thus, by \eqref{eqn_BGSEDF_BSWEDF}, we have $\lambda \leq \sum_{1\leq i\leq m}
\frac{\lambda_i\widetilde{k}}{k_i}$, i.e., $\cB$
is an $(n,m,(k_1,k_2,\cdots,k_m),a=\sum_{1\leq i\leq m}k_i,\lambda)$-BSWEDF.
\end{proof}

\section{Constructions of optimal BSWEDFs and SWEDFs}\label{sec-construction}
In this section, we are going to construct BSWEDFs and SWEDFs, which are generally not EDFs, or GSEDFs, or PEDFs.

We recall a well-known construction of difference families. Let $q=4k+1$ be
a prime power. Let $\alpha$ be a primitive element of $\F_q$,
\begin{equation}\label{eqn_D_2}
D^{2}_i=\{\alpha^{i+2j}\,:\,0\leq j\leq 2k-1\},\,\,\text{for}\,\,i=0,1
\end{equation}
and
\begin{equation}\label{eqn_D_4}
D^{4}_i=\{\alpha^{i+4j}\,:\,0\leq j\leq k-1\},\,\,\text{for}\,\,0\leq i\leq 3.
\end{equation}
It is well-known that $\{D^{2}_0,D^{2}_1\}$ is a $(q,2k,2k-1)$-DF over the additive group of $\F_q$.

\begin{construction} Let $\cS=\{S_1,S_2,S_3\}$ be the family of disjoint subsets $\Z_{2}\times\F_q$ defined as
\begin{equation*}
  S_1=\{(0,0),(1,0)\},\,\, S_2=\{0\}\times D^{4}_0\cup \{1\}\times D^{4}_2,\text{ and }  S_3=\{0\}\times D^{4}_1\cup \{0\}\times D^{4}_3.
  \end{equation*}
\end{construction}

\begin{theorem}\label{theorem_SWEDF}
Let $\cS=\{S_1,S_2,S_3\}$ be the family defined in Construction A. If $k$ is odd,
then $\cS$ is an optimal $(n=2q,m=3,(2,2k,2k),a=4k+2,\lambda=2k+1)$-BSWEDF.
\end{theorem}

Before the proof we list a well-known result about $D^{2}_0$ and $D^{2}_1$.

\begin{lemma}\label{lemma_D_4}
If $k$ is odd, then the family $\{D^{2}_0,
D^{2}_1\}$ satisfies
\begin{equation*}
D\left(D^{2}_0,D^{2}_1\right)\cup D\left(D^{2}_1,D^{2}_0\right)=2k\boxtimes (\F_q\backslash\{0\})
\end{equation*}
and
\begin{equation*}
D\left(D^{4}_0,D^{4}_1\right)\cup D\left(D^{4}_0,D^{4}_3\right)\cup
D\left(D^{4}_1,D^{4}_0\right)\cup D\left(D^{4}_3,D^{4}_0\right)=k\boxtimes (\F_q\backslash\{0\}).
\end{equation*}
\end{lemma}
\begin{proof}
By \eqref{eqn_D_2} and \eqref{eqn_D_4}, we have $D^{2}_0=D^{4}_0\cup D^{4}_2=D^{4}_0\cup (-D^{4}_0)$
and $D^{2}_1=D^{4}_1\cup D^{4}_3=D^{4}_1\cup (-D^{4}_1)$, where $\alpha^{2k}=-1$.
The fact $\{D^2_0,D^2_1\}$ is a $(q,2k,2k-1)$-PDF means that
\begin{equation*}
D\left(D^{2}_0,D^{2}_1\right)\cup D\left(D^{2}_1,D^{2}_0\right)=2k\boxtimes (\F_q\backslash\{0\}).
\end{equation*}
The preceding equality can be rewritten as
\begin{equation*}
\begin{split}
2k\boxtimes (\F_{q}\backslash\{0\})=&D\left(D^{2}_0,D^{2}_1\right)\cup D\left(D^{2}_1,D^{2}_0\right)\\
=&D\left(D^{4}_0\cup (-D^{4}_0),D^{4}_1\cup D^{4}_3\right)\cup D\left(D^{4}_1\cup D^{4}_3,D^{4}_0\cup (-D^{4}_0)\right)\\
=&2\boxtimes \left(D\left(D^{4}_0,D^{4}_1\right)\cup D\left(D^{4}_0,D^{4}_3\right)\cup
D\left(D^{4}_1,D^{4}_0\right)\cup D\left(D^{4}_3,D^{4}_0\right)\right),
\end{split}
\end{equation*}
where for the last equality we use the facts $D(-D^{4}_0,D^{4}_1\cup D^{4}_3)=D(-(D^{4}_1\cup D^{4}_3),D^{4}_0)=D(D^{4}_3\cup D^{4}_1,D^{4}_0)$
and $D\left(D^{4}_1\cup D^{4}_3, -D^{4}_0\right)=D\left(D^4_0,-(D^{4}_1\cup D^{4}_3)\right)=D\left(D^4_0,D^{4}_3\cup D^{4}_1\right)$.
This completes the proof.
\end{proof}

{\textit{Proof of Theorem \ref{theorem_SWEDF}:}}
By Definition \ref{def_BSWEDF}, in this case, $\widetilde{k}=\text{lcm}(2k,2)=2k$,
$\widetilde{S}_1=k\boxtimes \{(0,0),(1,0)\}$,
$\widetilde{S}_2=S_2$, and $\widetilde{S}_3=S_3$.
Thus, $D(S_2,\widetilde{S}_3)=D(S_2,{S_3})$ and
$D(S_3,\widetilde{S}_2)=D(S_3,S_2)$. Recall that $S_2=\{0\}\times D^4_0\cup \{1\}\times(-D^4_0)$,
which implies
\begin{equation}\label{eqn_D_0_D_1}
\begin{split}
&D(S_2,\widetilde{S}_3)\cup D(S_3,\widetilde{S}_2)\\
=&D(\{0\}\times D^4_0\cup \{1\}
\times(-D^4_0),\{0\}\times D^4_1 \cup \{0\}\times D^4_3)\\
&\cup D(\{0\}\times D^4_1\cup \{0\}\times D^4_3,\{0\}\times D^4_0\cup \{1\}\times(-D^4_0))\\
=&\bigcup_{i=0,1}\{i\}\times \left(D\left(D^{4}_0,D^{4}_1\right)\cup D\left(D^{4}_0,D^{4}_3\right)\cup
D\left(D^{4}_1,D^{4}_0\right)\cup D\left(D^{4}_3,D^{4}_0\right)\right)\\
=&k\boxtimes\left(\Z_2\times (\F_{q}\backslash\{0\})\right),
\end{split}
\end{equation}
where we use the fact $D^4_1=-D^4_3$ and the last equality holds by Lemma \ref{lemma_D_4}.
By the fact $\bigcup_{0\leq i\leq 3}D^4_i=\F_q\backslash\{0\}$, we have
\begin{equation*}
\begin{split}
D(S_1,\widetilde{S}_2)\cup D(S_2,\widetilde{S}_1)
=&\{0\}\times D^4_2 \cup \{1\}\times D^4_0 \cup \{1\}\times D^4_2 \cup \{0\}\times D^4_0 \\
&\cup k\boxtimes\left( \{0\}\times D^4_2\cup \{1\}\times D^4_0 \cup \{1\}\times D^4_2\cup \{0\}\times D^4_0\right)\\
=&(k+1)\boxtimes\left(\Z_2\times D^2_{0}\right)
\end{split}
\end{equation*}
and
\begin{equation*}
\begin{split}
D(S_1,\widetilde{S}_3)\cup D(S_3,\widetilde{S}_1)
=&\{0\}\times D^2_1\cup \{1\}\times D^2_1\cup k\boxtimes\left( \{0\}\times D^2_1\cup \{1\}\times D^2_1\right)\\
=&(k+1)\boxtimes\left(\Z_2\times D^2_{1}\right),
\end{split}
\end{equation*}
where we use the facts $D^2_i=D^4_i\cup D^4_{i+2}$ and $D^4_i=-D^4_{i+2}$ for $i=0,1$.
The above two equalities imply that
\begin{equation}\label{eqn_S_1_S_2}
\bigcup_{i=2,3} \left(D(S_1,\widetilde{S}_i)\cup D(S_i,\widetilde{S}_1)\right)=(k+1)\boxtimes\left(\Z_2\times(\F_q\backslash \{0\})\right).
\end{equation}
Therefore, by \eqref{eqn_D_0_D_1} and \eqref{eqn_S_1_S_2},
$$\bigcup_{1\leq i\ne j\leq 3}D(S_i,\widetilde{S}_j)=(2k+1)\boxtimes
\left(\Z_2\times(\F_q\backslash \{0\})\right)\subseteq (2k+1)\boxtimes \left((\Z_2\times\F_q)\backslash \{(0,0)\}\right),$$
i.e., $\cS=\{S_1,S_2,S_3\}$ is an $(n=2q,m=3,(2,2k,2k),a=4k+2,\lambda=2k+1)$-BSWEDF.
By Lemma \ref{lemma_bound_BSWEDF}, we have $$\lambda\geq \left\lceil\frac{\widetilde{k}a(m-1)}{n-1}\right\rceil=\left\lceil\frac{2k(4k+2)2}{2q-1}\right\rceil=\left\lceil\frac{2k(8k+1)+6k}{8k+1}\right\rceil=2k+1.$$
Thus, $\cS$ is an optimal $(n=2q,m=3,(2,2k,2k),a=4k+2,\lambda=2k+1)$-BSWEDF.

\qed

It is easily seen from the proof of Theorem \ref{theorem_SWEDF} that the above BSWEDFs are not EDFs, or GSEDFs, or PEDFs.
\begin{example}
Let $n=2q=26$. By Construction A, the family of sets $\cS=\{S_1,S_2,S_3\}$ over $\Z_{26}$
can be listed as
\begin{equation*}
  S_1=\{0,13\},\,\, S_2=\{14,16,22,17,25,23\},\,\, \text{and}\,\, S_3=\{2,6,18,8,24,20\}.
\end{equation*}
It is easy to check that
\begin{equation*}
\bigcup_{1\leq i\ne j\leq 3}D(S_i,\widetilde{S}_j)=7\boxtimes (\Z_{26}\backslash\{0,13\}),
\end{equation*}
which means that $\cS$ is an optimal $(26,3,(2,6,6),14,7)$-BSWEDF.
\end{example}

Let $n_1=2k+1$ and $\{\{0\},E_1,E_2\}$ be an $(n_1,k,k-1)$-PDF over
an Abelian group $G$ of order $n_1$.
Such kinds of PDFs exist, for example, when $n_1$ is a prime power, and $E_1=D^2_0,\,E_2=D^2_1$.
Based on
$\{\{0\},E_1,E_2\}$ we can construct a BSWEDF as follows.

\begin{construction}
Let $\cW=\{W_1,W_2,W_3\}$ be the family of disjoint subsets of $\Z_2\times G$,
defined as  $W_1=\{(1,0)\}$, $W_2=\{0\}\times E_1$, and $W_3=\{0\}\times E_2$.
\end{construction}

\begin{theorem}\label{theorem_cons_B}
The family $\cW=\{W_1,W_2,W_3\}$ generated by Construction B is an
optimal $(n=2n_1,3,(1,k,k),2k+1,k+1)$-BSWEDF.
\end{theorem}
\begin{proof}
The fact that $\{\{0\},E_1,E_2\}$ is an $(n_1=2k+1,k,k-1)$-PDF means that
$D(E_1,E_2)\cup D(E_2,E_1)=k\boxtimes (G\backslash\{0\})$.
Thus, we have
\begin{equation*}
D(W_2,\widetilde{W}_3)\cup D(W_3,\widetilde{W}_2)=D(W_2,W_3)\cup D(W_3,W_2)=k \boxtimes (\{0\}\times(G\backslash\{0\})),
\end{equation*}
where we apply the fact $\widetilde{k}=\text{lcm}(1,k,k)=k=|W_2|=|W_3|$.
Note that
\begin{equation*}
\begin{split}
&D(W_1,\widetilde{W}_2)\cup D(W_1,\widetilde{W}_3)\cup D(W_3,\widetilde{W}_1)\cup D(W_2,\widetilde{W}_1)\\
=&\{1\}\times (-E_1)\cup \{1\}\times (-E_2)\cup D(\{0\}\times E_1,k\boxtimes \{(1,0)\})\cup D(\{0\}\times E_2,k\boxtimes \{(1,0)\})\\
=&(k+1)\boxtimes (\{1\}\times (G\backslash \{0\})).
\end{split}
\end{equation*}
Based on the above two equalities,
\begin{equation*}
\bigcup_{1\leq i\ne j\leq 3} D(W_i,\widetilde{W}_j)\subseteq (k+1)\boxtimes((\Z_2\times G)\backslash\{(0,0)\} ),
\end{equation*}
i.e., $\cW$ is an $(n=2n_1,m=3,(1,k,k),a=2k+1,\lambda=k+1)$-BSWEDF.

By Lemma \ref{lemma_bound_BSWEDF}, we have $$\lambda\geq \left\lceil\frac{\widetilde{k}a(m-1)}{n-1}\right\rceil=\left\lceil\frac{k(2k+1)2}{2n_1-1}\right\rceil=\left\lceil\frac{k(4k+1)+k}{4k+1}\right\rceil=k+1.$$
Thus, $\cW$ is an optimal $(2n_1=4k+2,3,(1,k,k),2k+1,k+1)$-BSWEDF.

\end{proof}

It is easily seen from the proof of Theorem \ref{theorem_cons_B} that the above BSWEDFs are not EDFs, or GSEDFs, or PEDFs.

\begin{example}
Let $n=2n_1=22$. By Construction B, the family of sets $\cW=\{W_1,W_2,W_3\}$ over $\Z_{22}$
can be listed as
\begin{equation*}
  W_1=\{11\},\,\, W_2=\{12, 4, 16, 20, 14\},\,\, \text{and}\,\, W_3=\{2, 8, 10, 18, 6\}.
\end{equation*}
It is easy to check that
\begin{equation*}
\bigcup_{1\leq i\ne j\leq 3}D(W_i,\widetilde{W}_j)\subseteq 6\boxtimes (Z_{22}\backslash\{0\}),
\end{equation*}
which means that $\cW$ is an optimal $(22,3,(1,5,5),11,6)$-BSWEDF.
\end{example}

\begin{construction}
Let $q=4k+1$ be a prime power and let $\cU=\{U_1,U_2,U_3,U_4\}$ be the family of disjoint subsets of $\Z_3\times\F_q$,
defined as  $U_1=\{(1,0)\}$, $U_2=\{(2,0)\}$, $U_3=\{0\}\times D^2_0$, and
$U_4=\{0\}\times D^2_1$.
\end{construction}

\begin{theorem}\label{theorem_cons_C}
The family $\cU=\{U_1,U_2,U_3,U_4\}$ in Construction C is an
optimal $(3q=12k+3,4,(1,1,2k,2k),4k+2,2k+1)$-BSWEDF.
\end{theorem}
\begin{proof}
Note that $\widetilde{k}=\text{lcm}(1,1,2k,2k)=2k$, which implies $\widetilde{U}_3=U_3$
and $\widetilde{U}_4=U_4$.
Lemma \ref{lemma_D_4} shows that
$D(D^{2}_0,D^{2}_1)\cup D(D^{2}_1,D^{2}_0)=2k\boxtimes (\F_q\backslash\{0\})$.
Thus, we have
\begin{equation*}
D(U_3,\widetilde{U}_4)\cup D(U_4,\widetilde{U}_3)=D(U_3,U_4)\cup D(U_3,U_4)=2k \boxtimes (\{0\}\times(\F_q\backslash\{0\})).
\end{equation*}

Recall that
\begin{equation*}
\begin{split}
&D(U_1,\widetilde{U}_3)\cup D(U_1,\widetilde{U}_4)\cup D(U_3,\widetilde{U}_1)\cup D(U_4,\widetilde{U}_1)\\
=&(\{1\}\times D^2_0)\cup (\{1\}\times D^2_1)\cup D(\{0\}\times D^2_0,2k\boxtimes \{(1,0)\})\cup D(\{0\}\times D^2_1,2k\boxtimes \{(1,0)\})\\
=&(\{1\}\times (\F_q\backslash \{0\}))\cup 2k\boxtimes (\{2\}\times (\F_q\backslash \{0\}))
\end{split}
\end{equation*}
and
\begin{equation*}
\begin{split}
&D(U_2,\widetilde{U}_3)\cup D(U_2,\widetilde{U}_4)\cup D(U_3,\widetilde{U}_2)\cup D(U_4,\widetilde{U}_2)\\
=&\{2\}\times D^2_0\cup \{2\}\times D^2_1\cup D(\{0\}\times D^2_0,2k\boxtimes \{(2,0)\})\cup D(\{0\}\times D^2_1,2k\boxtimes \{(2,0)\})\\
=&(\{2\}\times (\F_q\backslash \{0\}))\cup 2k\boxtimes (\{1\}\times (\F_q\backslash \{0\})).
\end{split}
\end{equation*}
For the differences between $U_1$ and $U_2$, we have
$$D(U_1,\widetilde{U}_2)\cup D(U_2,\widetilde{U}_1)=2k\boxtimes\{(1,0),(2,0)\}.$$

Therefore, the above four equalities mean that
\begin{equation*}
\begin{split}
&\bigcup_{1\leq i\ne j\leq 4} D(U_i,\widetilde{U}_j)\\
=&(2k\boxtimes\{(1,0),(2,0)\})\cup (2k\boxtimes \{0\}\times (\F_q\backslash\{0\}))\cup ((2k+1)\boxtimes \{1,2\}\times (\F_q\backslash\{0\}))\\
\subseteq& (2k+1)\boxtimes((\Z_3\times\F_q)\backslash\{(0,0)\} ),
\end{split}
\end{equation*}
i.e., $\cU$ is an $(n=3q,m=4,(1,1,2k,2k),a=4k+2,\lambda=2k+1)$-BSWEDF.

By Lemma \ref{lemma_bound_BSWEDF}, we have $$\lambda\geq \left\lceil\frac{\widetilde{k}a(m-1)}{n-1}\right\rceil=\left\lceil\frac{2k(4k+2)3}{3q-1}\right\rceil=\left\lceil\frac{2k(12k+2)+8k}{12k+2}\right\rceil=2k+1.$$
Thus, $\cU$ is an optimal $(3q,4,(1,1,2k,2k),4k+2,2k+1)$-BSWEDF.

\end{proof}

It is easily seen from the proof of Theorem \ref{theorem_cons_C} that the above BSWEDFs are not EDFs, or GSEDFs, or PEDFs.

\begin{example}
Let $n=3q=39$. By Construction A, the family of sets $\cU=\{U_1,U_2,U_3,U_4\}$ over $\Z_{39}$
can be listed as
\begin{equation*}
  U_1=\{13\},\,\, U_2=\{26\},\,\, U_3=\{27, 30, 3, 12, 9, 36\},\,\,\text{and}\,\, U_4=\{15, 21, 6, 24, 18, 33\}.
\end{equation*}
It is easy to check that
\begin{equation*}
\bigcup_{1\leq i\ne j\leq 4}D(U_i,\widetilde{U}_j)\subseteq 7\boxtimes (\Z_{39}\backslash\{0\}),
\end{equation*}
which means that $\cU$ is an optimal $(39,4,(1,1,6,6),14,7)$-BSWEDF.
\end{example}

\subsection{A construction of cyclic SWEDFs}
In this subsection, we are going to construct cyclic SWEDFs, which are not
regular EDFs, or GSEDFs, or PEDFs. A \textit{cyclic} SWEDF means an SWEDF over a cyclic additive group.

A well-studied kind of PDFs $\cR=\{R_1,R_2,\cdots,R_l\}$ are those
with parameters $(n=(k-1)(tk+1),(k,\cdots,k,k-1),k-1)$
over $\Z_{n}=\Z_{k-1}\times\Z_{tk+1}$ where ${\rm gcd}(k-1,tk+1)=1$,
$R_l=\Z_{k-1}\times\{0\}$ and $l=t(k-1)+1$.
In Table \ref{tab PDF}, we list such PDFs which can be applied in the following
construction.

\begin{table}
\centering
\begin{threeparttable}[b]
\caption{Some known PDFs with parameters $(n,\mathcal{W}=(k^{\frac{n-k+1}{k}},(k-1)^1),k-1)$\label{tab PDF}} \center
 \begin{tabular}{|c|c|c|}
\hline
Parameters & Constraints  &Ref.\\
\hline
\hline $\left(2v,\,(3^{\frac{2v-2}{3}},2^1),\,2 \right)$, & $\begin{array}{c}v=p_1^{m_1}p_2^{m_2}\cdots p_r^{m_r},\,2<p_1<p_2<\cdots<p_r,\\
{\rm and }\,\,3|(p_{t}-1)\,{\rm for }\, 1\leq t\leq r\end{array}$&{\cite{BYW2010}}\\
\hline $\left(sv,\,((s+1)^\frac{sv-s}{s+1},s^1),\,s \right)$& $\begin{array}{c}v=p_1^{m_1}p_2^{m_2}\cdots p_r^{m_r},\,2<p_1<p_2<\cdots<p_r,\\
{\rm and }\,\,2(s+1)|(p_{t}-1)\,{\rm for }\, 1\leq t\leq r, s=4,5 \end{array}$&{\cite{BYW2010}}\\
\hline $\left(6v,\,(7^\frac{6v-6}{7},6^1),\,6 \right)$& $\begin{array}{c}v=p_1^{m_1}p_2^{m_2}\cdots p_r^{m_r},\,2<p_1<p_2<\cdots<p_r,\\
{\rm and }\,\,28|(p_{t}-1)\,{\rm for }\, 1\leq t\leq r \end{array}$&{\cite{BYW2010}}\\
\hline $\left(7v,\,(8^\frac{7v-7}{8},7^1),\,7 \right)$& $\begin{array}{c}v=p_1^{m_1}p_2^{m_2}\cdots p_r^{m_r},\,2<p_1<p_2<\cdots<p_r,\\
{\rm and }\,\,8|(p_{t}-1)\,{\rm for }\, 1\leq t\leq r, v\not\in \{17, 89\} \end{array}$&{\cite{BYW2010}}\\
\hline $(q-1,(\frac{q}{d}^{d-1},(\frac{q}{d}-1)^1),{q-d\over d})$&$d|q$, $\text{gcd}(\frac{q}{d}-1,(q-1)/(\frac{q}{d}-1))=1$&{\cite{D2008}}\\
\hline
\end{tabular}
\begin{tablenotes}
\item[] {Herein $p_i$'s are primes; $t$, $s$, $r$ and $m$ are positive integers; $q$ is a prime power.}
\end{tablenotes}
\end{threeparttable}
\end{table}

\begin{construction}\label{cons_SWEDF}
Let $\cV=\{V_1,V_2,\cdots,V_{t(k-1)+k-2}\}$ be the  family of disjoint subsets of $\Z_n$,
defined as
  $$V_i=R_i \quad \text{for }1\leq i\leq t(k-1),$$
  \begin{equation*}
  V_{t(k-1)+j}=\{(j,0)\}\text{ for }1\leq j\leq k-2.
  \end{equation*}
\end{construction}

\begin{theorem}\label{theorem_cons_SWEDF}
Let $\cV$ be the family in Construction D. Then $\cV$
is a cyclic  $(n,t(k-1)+k-2,K=(k,\cdots,k,1,1,\cdots,1),n-1,(t+1)k^2-(t+3)k)$-SWEDF,
where the element $1$ appears $k-2$ times and the element $k$ appears $t(k-1)$
times in $K$.
\end{theorem}
\begin{proof}
Since $\cR$ is an $(n=(k-1)(tk+1),(k,\cdots,k,k-1),k-1)$ PDF, we can conclude that
\begin{equation*}
\bigcup_{1\leq i\ne j\leq l} D(R_i,R_j)=(n-k+1)\boxtimes ((\Z_{k-1}\times \Z_{tk+1})\backslash\{(0,0)\}).
\end{equation*}
Recall that $R_l=\Z_{k-1}\times\{0\}$, which means
\begin{equation*}
\bigcup_{1\leq i\leq l-1}(D(R_i,R_l)\cup D(R_l,R_i))=(2k-2)\boxtimes (\Z_{k-1}\times(\Z_{tk+1}\backslash\{0\})).
\end{equation*}
Thus, by Construction D, we have
\begin{equation}\label{eqn_V_diff_1}
\begin{split}
&\bigcup_{1\leq i\ne j\leq l-1}D(V_i,\widetilde{V}_j)=\bigcup_{1\leq i\ne j\leq l-1}D(V_i,V_j)=
\bigcup_{1\leq i\ne j\leq l-1}D(R_i,R_j)\\
=&\left(\bigcup_{1\leq i\ne j\leq l} D(R_i,R_j)\right)\backslash \left(\bigcup_{1\leq i\leq l-1}(D(R_i,R_l)\cup D(R_l,R_i))\right)\\
=&\left((n-k+1)\boxtimes((\Z_{k-1}\backslash\{0\})\times \{0\})\right)\cup \left((n-3k+3)\boxtimes (\Z_{k-1}\times(\Z_{tk+1}\backslash\{0\}))\right),
\end{split}
\end{equation}
where we use the fact $\widetilde{k}=k$.

Note that for any $1\leq j\leq k-2$,
\begin{equation*}
\begin{split}
&\bigcup_{1\leq i\leq l-1}(D(V_i,\widetilde{V}_{l-1+j})\cup D(V_{l-1+j},\widetilde{V}_i))\\
=&\bigcup_{1\leq i\leq l-1}(D(R_i,k\boxtimes \{(j,0)\})\cup D(\{(j,0)\},R_i))=
(k+1)\boxtimes (\Z_{k-1}\times(\Z_{tk+1}\backslash\{0\})).
\end{split}
\end{equation*}
Thus, we have
\begin{equation}\label{eqn_V_diff_2}
\bigcup_{1\leq j\leq k-2}\bigcup_{1\leq i\leq l-1}(D(V_i,\widetilde{V}_{l-1+j})\cup D(V_{l-1+j},\widetilde{V}_i))=
((k+1)(k-2))\boxtimes (\Z_{k-1}\times(\Z_{tk+1}\backslash\{0\})).
\end{equation}
For the last part of external differences, we have
\begin{equation}\label{eqn_V_diff_3}
\begin{split}
\bigcup_{1\leq i\ne j\leq k-2}D(V_{l-1+i},\widetilde{V}_{l-1+j})
=&\bigcup_{1\leq i\ne j\leq k-2}D(\{(i,0)\},k\boxtimes\{(j,0)\})\\
=&k\boxtimes\left(\bigcup_{1\leq i\ne j\leq k-2}D(\{(i,0)\},\{(j,0)\})\right)\\
=&k(k-3)\boxtimes ((\Z_{k-1}\backslash\{0\})\times \{0\}).
\end{split}
\end{equation}
Combining \eqref{eqn_V_diff_1}, \eqref{eqn_V_diff_2} and \eqref{eqn_V_diff_3},
\begin{equation*}
\begin{split}
&\bigcup_{1\leq i\ne j\leq l+k-3}D(V_i,\widetilde{V}_j)\\
=&\left(\bigcup_{1\leq i\ne j\leq l-1}D(V_i,\widetilde{V}_j)\right)\cup
\left(\bigcup_{1\leq j\leq k-2}\bigcup_{1\leq i\leq l-1}(D(V_i,\widetilde{V}_{l-1+j})\cup D(V_{l-1+j},\widetilde{V}_i))\right)\cup
\left(\bigcup_{1\leq i\ne j\leq k-2}D(V_{l-1+i},\widetilde{V}_{l-1+j})\right)\\
=&\left((n-k+1+k(k-3))\boxtimes (\Z_{k-1}\backslash\{0\})\times \{0\}\right)\cup
\left((n-3k+3+(k+1)(k-2))\boxtimes (\Z_{k-1}\times(\Z_{tk+1}\backslash\{0\}))\right)\\
=&((t+1)k^2-tk-3k)\boxtimes ((\Z_{k-1}\times \Z_{tk+1})\backslash\{(0,0)\}),
\end{split}
\end{equation*}
where $n=(k-1)(tk+1)$.

Therefore, $\cV$ is a cyclic $(n,t(k-1)+k-2,(k,k,\cdots,k,1,1,\cdots,1),n-1,(t+1)k^2-(t+3)k)$-SWEDF,
where the element $1$ occurs $k-2$ times in $K$ and the element $k$ appears $t(k-1)$
times in $K$. This completes the proof.

\end{proof}

It is easily seen from the proof of Theorem \ref{theorem_cons_SWEDF} that the above SWEDFs are not regular EDFs, or GSEDFs, or PEDFs.

In \cite{HP}, Huczynska and Paterson introduced some constructions of SWEDFs (or equivalently, RWSEDs)
with the so-called bimodal property.

\begin{definition}[\cite{HP}]\label{def_bimodal}
Let $G$ be a finite Abelian group and $\cB$ be a collection $B_1, B_2,\dots, B_m$
of disjoint subsets of $G$ with sizes $k_1, k_2,\dots,k_m$, respectively.
We say that $\cB$ has the \textit{bimodal property} if for each $\delta\in G\backslash\{0\}$
we have $N_i(\delta)\in\{0, k_i\}$ for $1\leq i\leq m$, where
$N_i(\delta)$
is defined in Definition \ref{def_RWEDF}.
\end{definition}

The SWEDF generated by Construction D does not have the  bimodal property.
Let $\cV$ be the SWEDF generated by Construction D. For any $v\in V_i$ with
$|V_i|=k$,
we have $0\in D(V_i,\{v\})$ and $|D(V_i,\{v\})|=|V_i|=k$. However, by Construction D,
$0$ is not an element of $V_j$ for $1\leq j\leq l+k-3$. Thus, the number
of solutions for $a-b=v$ for $a\in V_i$ and $b\in V_j$ for $1\leq j\leq l+k-3$
and $j\ne i$ is at most $k-1$, since $\bigcup_{1\leq j\leq l+k-3}V_j=\Z_n\backslash\{0\}$, i.e.,
$N_i(v)\leq k-1$.
Next, we show that there exists $V_i$ with $|V_i|=k$ satisfying $N_i(v)\ne 0$.
If $a-b\ne v$ for all $a\in V_i$ and $b\in V_j$ for $1\leq j\leq l+k-3$
and $j\ne i$,  then $a\in V_i$ means that $(a+\langle v\rangle)\backslash\{0\}\subseteq V_i$.
This is to say that $V_i$ is the union of some cosets of $\langle v\rangle$ besides the
element $0$ and $k=\tau|\langle v\rangle|-1$ for some integer $\tau\geq 1$.
This is impossible since there are elements $v$ with
$|\langle v\rangle|>k+1$ in $\Z_n\backslash\{0\}$. Thus, the SWEDF generated by
Construction D is not bimodal. For more details about SWEDFs (or equivalently, RWEDFs) with bimodal property
the reader may refer to \cite{HP,HP2019}.

Compared with the constructions in \cite{HP}, Construction \ref{cons_SWEDF}
can generate RWEDFs with flexible parameters without bimodal property. To the
best of our knowledge, this is the first class of RWEDFs without the bimodal property,
which are not regular EDFs, or GSEDFs, or PEDFs.

\begin{corollary}
Let $\cV$ be the family in Construction D. Then $\cV$
is an $(n,t(k-1)+k-2,K=(k,\cdots,k,1,1,\cdots,1),n-1,(t+1)k-t-3)$-RWEDF without
the bimodal property,
where the element $1$ appears $k-2$ times and the element $k$ appears $t(k-1)$
times in $K$.
\end{corollary}

\begin{example}
Let $G=(\Z_{15},+)$ and $\cR=\{R_1=\{6,9,2,8\},R_2=\{11,14,7,13\},R_3=\{1,4,12,3\}, R_4=\{0,5,10\}\}$.
It is easy to check
that $\cR$ is a PDF with parameters $(15,(4,4,4,3),3)$. By Construction D, we generate a family of subsets of $\Z_{15}$
as
 $\cV=\{V_1=\{6,9,2,8\},V_2=\{11,14,7,13\},V_3=\{1,4,12,3\},V_4=\{5\},V_5=\{10\}\}$. It is easy to check
 that
 $$\bigcup_{1\leq i\ne j\leq 5}D(V_i,\widetilde{V}_j)=16\boxtimes (\Z_{15}\backslash\{0\}),$$
 i.e., $\cV$ is a $(15,5,(4,4,4,1,1),14,16)$-SWEDF (or $(15,5,(4,4,4,1,1),14,4)$-RWEDF).
 Note that $N_3(6)=3\not\in\{0,4\}$,
which means the SWEDF does not have the bimodal property by Definition \ref{def_bimodal}.
\end{example}

\section{Concluding Remarks}\label{sec-conclusion}
In this paper, we first characterized weak algebraic manipulation detection
codes via bounded standard weighted external difference families (BSWEDFs). As
a byproduct, we improved the known lower bound for weak algebraic manipulation
detection codes. To generate optimal weak AMD codes, constructions for BSWEDFs, especially, a
construction of SWEDFs without the bimodal property, were introduced.

Combinatorial structures, e.g., BSWEDFs, SWEDFs, strong external difference
families (SEDFs), partitioned external difference families (PEDFs), play a key
role in the constructions of weak algebraic manipulation detection (AMD) codes.
There are some known results for the existence of SEDFs. However, the
existence of BSWEDFs, SWEDFs, and PEDFs are generally open.
Finding more explicit constructions for such combinatorial structures are not only
an interesting subject for AMD codes but also an interesting problem in their
own right, which is left for future research.

\section*{acknowledgements}
The authors would like to thank Prof. Marco Buratti for the helpful discussion
about difference families. This research is supported by JSPS
Grant-in-Aid for Scientific Research (B) under Grant No. 18H01133.


\begin{thebibliography}{11}

\bibitem{AS} H. Ahmadi and R. Safavi-Naini, ``Detection of algebraic manipulation in the presence of leakage,"
\emph{ICITS 2013,} Lecture Notes in Computer Science, vol. 8317, pp. 238-258, 2013.

\bibitem{BJWZ}J. Bao, L. Ji, R. Wei, and Y. Zhang, ``New existence and nonexistence results for strong external difference families,"
\emph{Discrete Mathematics,} vol. 341, no. 6, pp. 1798-1805, 2018.

\bibitem{BYW2010} M. Buratti, J. Yan, and C. Wang, ``From a $1$-rotational RBIBD to a partitioned difference family,"  \emph{Electron. J. Comb.,} vol. 17, R139, 2010.


\bibitem{CD} C.J.  Colbourn and J.H. Dinitz, \emph{Handbook of Combinatorial Designs}, Chapman
\& Hall/CRC, vol. 42, 2006.



\bibitem{CDFPW} R. Cramer, Y. Dodis, S. Fehr, C. Padr\'{o}, and D. Wichs, ``Detection of algebraic manipulation
with applications to robust secret sharing and fuzzy extractors,"  \emph{EUROCRYPT 2008},
Lecture Notes in Computer Science, vol 4965, pp. 471-488, 2008.

\bibitem{CFP} R. Cramer, S. Fehr, and C. Padr\'{o}, ``Algebraic manipulation codes," \emph{Science China Mathematics},
vol. 56, no. 7, pp. 1349-1358, 2013.

\bibitem{CPX} R. Cramer. C. Padr\'{o}, and C. Xing, ``Optimal algebraic manipulation detection codes in the
constant-error model," \emph{TCC2015}, Lecture Notes in Computer Science,  vol. 9014, pp. 481-501, 2015.



\bibitem{D2008} C. Ding, ``Optimal constant composition codes from zero-difference balanced functions," \emph{IEEE Trans. Inf. Theory,} vol. 54, no. 12, pp. 5766-5770, 2008.


\bibitem{D2009} C. Ding, ``Optimal and perfect difference systems of sets," \emph{Journal of Combinatorial Theory, Series A,} vol. 116, no. 1, pp. 109-119,  2009.

\bibitem{FG} C. Fan and G. Ge, ``A unified approach to Whiteman's and Ding-Helleseth's
generalized cyclotomy over residue class rings," \emph{IEEE Trans. Inf. Theory,} vol. 60, no. 2, pp. 1326-1336, 2014.

\bibitem{FT} Y. Fujiwara and V. D. Tonchev, ``High-rate self-synchronizing codes," \emph{IEEE Trans. Inf. Theory,} vol. 59, no. 4, pp. 2328-2335, 2012.


\bibitem{HP2018} S. Huczynska and M. B. Paterson, ``Existence and non-existence results for strong external difference families," \emph{Discrete Mathematics,}
vol. 341, no. 1, pp. 87-95, 2018.

\bibitem{HP} S. Huczynska and M. B. Paterson,  ``Weighted external difference families and $R$-optimal AMD codes," \emph{Discrete Mathematics,}
vol 342, no. 3, pp. 855-867, 2019.

\bibitem{HP2019} S. Huczynska and M. B. Paterson, ``Characterising bimodal collections of sets in finite groups," \emph{arXiv:} 1903.11620, 2019.

\bibitem{JL} J. Jedwab and S. Li, ``Construction and nonexistence of strong external difference families," \emph{Journal of Algebraic Combinatorics,}
vol. 49, no .1, pp. 21-48, 2019.

\bibitem{L} V. I. Levenshtein,  ``Combinatorial problems motivated by comma-free codes," \emph{Journal of Combinatorial Designs}, vol. 12, no. 3,
pp. 184-196, 2004.

\bibitem{LNC} X. Lu, X. Niu, and H. Cao, ``Some results on generalized strong external difference families," \emph{Designs, Codes and Cryptography,}
vol. 86, no. 12, pp. 2857-2868, 2018.

\bibitem{MS} W. J. Martin and D. R. Stinson, ``Some nonexistence results for strong external difference families using character theory,"
\emph{Bull. Inst. Combin. Appl.,} vol. 80, pp. 79-92,  2017.

\bibitem{NP} S. L. Ng and M. B. Paterson, ``Disjoint difference families and their applications," \emph{Designs, Codes and Cryptography,} vol. 78, no. 1,
pp. 103-127, 2016.

\bibitem{PS}
M. B. Paterson and D. R. Stinson,  ``Combinatorial characterizations of algebraic manipulation detection codes involving generalized difference families,"
 \emph{Discrete Mathematics,} vol. 339, no. 12, pp. 2891-2906, 2016.



\bibitem{WYF} J. Wen, M. Yang, and  K. Feng, ``The $(n,m,k,\lambda)$-strong external difference family with $m\geq 5$ exists," \emph{arXiv:} 1612.09495v1, 2016.

\bibitem{WYFF} J. Wen, M. Yang, F. Fu, and  K. Feng, ``Cyclotomic construction of strong external difference families in finite fields," \emph{Designs, Codes and Cryptography,}
 vol. 86, no. 5, pp. 1149-1159, 2018.




\end{thebibliography}
\end{document}